\documentclass[a4paper]{amsart}
\usepackage{graphicx}
\usepackage{amsmath}
\usepackage{amssymb}
\usepackage{oldgerm}
\usepackage[all]{xy}

\newcommand{\Hom}{\operatorname{Hom}\nolimits}
\newcommand{\End}{\operatorname{End}\nolimits}

\renewcommand{\mod}{\operatorname{mod}\nolimits}

\newcommand{\gr}{\operatorname{gr}\nolimits}

\newcommand{\Pd}{\operatorname{pd}\nolimits}
\newcommand{\gldim}{\operatorname{gldim}\nolimits}
\newcommand{\repdim}{\operatorname{repdim}\nolimits}

\newtheorem{theorem}{Theorem}[section]

\newtheorem{corollary}[theorem]{Corollary}

\newtheorem{lemma}[theorem]{Lemma}
\newtheorem{proposition}[theorem]{Proposition}
\theoremstyle{definition}

\theoremstyle{definition}

\theoremstyle{definition}

\theoremstyle{definition}

\theoremstyle{definition}

\theoremstyle{definition}

\theoremstyle{remark}
\newtheorem*{remark}{Remark}
\theoremstyle{remark}
\newtheorem*{remarks}{Remarks}
\theoremstyle{definition}

\theoremstyle{definition}

\begin{document}
\title{The representation dimension of quantum complete intersections}
\author{Petter Andreas Bergh \& Steffen Oppermann}
\address{Petter Andreas Bergh \newline Institutt for matematiske fag \\
  NTNU \\ N-7491 Trondheim \\ Norway}
\email{bergh@math.ntnu.no}
\address{Steffen Oppermann \newline Institutt for matematiske fag \\
  NTNU \\ N-7491 Trondheim \\ Norway}
\email{opperman@math.ntnu.no}


\subjclass[2000]{16G60, 16S80, 16U80, 81R50}

\keywords{Representation dimension, quantum complete intersections}

\maketitle

\begin{abstract}
We give a lower and an upper bound for the representation dimension
of a quantum complete intersection.
\end{abstract}

\section{Introduction}\label{intro}

In \cite{Auslander1}, Auslander introduced the representation
dimension of an Artin algebra in order to study algebras of infinite
representation type. He showed that a non-semisimple algebra is of
finite type if and only if its representation dimension is exactly
two, whereas it is of infinite type if and only if the
representation dimension is at least three.

For more than three decades no general method for computing the
representation dimension was known, and it was even unclear if this
dimension could exceed three. A negative answer to the latter would
imply that the finitistic dimension conjecture holds (cf.\
\cite{Igusa}). However, in 2006 Rouquier showed in \cite{Rouquier2}
that the representation dimension of the exterior algebra on a
$d$-dimensional vector space is $d+1$, using the notion of the
dimension of a triangulated category (cf.\ \cite{Rouquier1}). In
particular, there exist finite dimensional algebras of arbitrarily
large representation dimension. Other examples illustrating this
were subsequently given in \cite{Krause} and \cite{Oppermann1}.

In this paper, we study the representation dimension of quantum
complete intersections, a class of algebras originating from work by
Manin (cf.\ \cite{Manin}) and Avramov, Gasharov and Peeva (cf.\
\cite{Avramov}). In particular, under some assumptions we show that
the representation dimension of such an algebra is strictly greater
than its codimension.

\section{Representation dimension}\label{repdim}

Let $\Lambda$ be an Artin algebra, and denote by $\mod \Lambda$ the
category of finitely generated $\Lambda$-modules. The
\emph{representation dimension} of $\Lambda$, denoted $\repdim
\Lambda$, is defined as
$$\repdim \Lambda \stackrel{\text{def}}{=} \inf \{ \gldim
\End_{\Lambda}(M) \mid M \text{ generates and cogenerates } \mod
\Lambda \},$$ where $\gldim$ denotes the global dimension of an
algebra. A priori we see that the representation dimension may be
infinite, but Iyama showed in \cite{Iyama} that it is finite for
every Artin algebra. Auslander himself showed that the
representation dimension of a selfinjective algebra is at most its
Loewy length.

When computing the representation dimension of the exterior
algebras, Rouquier used the notion of dimensions of triangulated
categories, a concept he introduced in \cite{Rouquier1}. As we will
also use this notion when computing a lower bound for the
representation dimension of a quantum complete intersection, we
recall the definitions. Let $\mathcal{T}$ be a triangulated
category, and let $\mathcal{C}$ and $\mathcal{D}$ be subcategories
of $\mathcal{T}$. We denote by $\langle \mathcal{C} \rangle$ the
full subcategory of $\mathcal{T}$ consisting of all the direct
summands of finite direct sums of shifts of objects in
$\mathcal{C}$. Furthermore, we denote by $\mathcal{C} \ast
\mathcal{D}$ the full subcategory of $\mathcal{T}$ consisting of
objects $M$ such that there exists a distinguished triangle
$$C \to M \to D \to C[1]$$
in $\mathcal{T}$, with $C \in \mathcal{C}$ and $D \in \mathcal{D}$.
Finally, we denote the subcategory $\langle \mathcal{C} \ast
\mathcal{D} \rangle$ by $\mathcal{C} \diamond \mathcal{D}$. Now
define $\langle \mathcal{C} \rangle_0$ and $\langle \mathcal{C}
\rangle_1$ to be $0$ and $\langle \mathcal{C} \rangle$,
respectively, and for each $n \ge 2$ define inductively $\langle
\mathcal{C} \rangle_n$ to be $\langle \mathcal{C} \rangle_{n-1}
\diamond \langle \mathcal{C} \rangle$. The \emph{dimension} of
$\mathcal{T}$, denoted $\dim \mathcal{T}$, is defined as
$$\dim \mathcal{T} \stackrel{\text{def}}{=} \inf \{ d \in
\mathbb{Z} \mid \text{ there exists an object } M \in \mathcal{T}
\text{ such that } \mathcal{T} = \langle M \rangle_{d+1} \}.$$

From the definition, we see that the dimension of a triangulated
category is not necessarily finite. However, the following
elementary result provides a method for computing an upper bound in
terms of dense subcategories.

\begin{proposition}\cite[Lemma 3.4]{Rouquier1}\label{densesubcat}
If $F \colon \mathcal{S} \to \mathcal{T}$ is a functor of
triangulated categories whose image is dense in $\mathcal{T}$, then
$\dim \mathcal{T} \le \dim \mathcal{S}$.
\end{proposition}

We will use this result to show that for any quantum complete
intersection of codimension $n$, there exists a chain of $n$
subalgebras and a corresponding chain of $n-1$ inequalities bounding
the representation dimension from below. When proving both this and
the main results, the triangulated category we use is the
\emph{stable module category} of the algebra. Recall therefore that
when $\Lambda$ is a selfinjective algebra, its stable module
category, denoted $\underline{\mod} \Lambda$, is defined as follows:
the objects of $\underline{\mod} \Lambda$ are the same as in $\mod
\Lambda$, but two morphisms in $\mod \Lambda$ are equal in
$\underline{\mod} \Lambda$ whenever their difference factors through
a projective $\Lambda$-module. The cosyzygy functor
$\Omega_{\Lambda}^{-1} \colon \underline{\mod} \Lambda \to
\underline{\mod} \Lambda$ is an equivalence of categories, and a
triangulation of $\underline{\mod} \Lambda$ is given by using this
functor as a shift and by letting short exact sequences in $\mod
\Lambda$ correspond to triangles. Thus $\underline{\mod} \Lambda$ is
a triangulated category, and its dimension is related to the
representation dimension of $\Lambda$ by the following result.

\begin{proposition}\cite[Proposition 3.7]{Rouquier2}\label{lowerbound}
If $\Lambda$ is a non-semisimple selfinjective algebra, then
$\repdim \Lambda \ge \dim ( \underline{\mod} \Lambda ) +2$.
\end{proposition}

We end this section with the following lemma. It gives a lower bound
for the dimension of the stable module category of a selfinjective
algebra, in terms of certain subalgebras.

\begin{lemma}\label{subalgebra}
\sloppy Let $A$ and $B$ be finite dimensional selfinjective
$k$-algebras. If there exist algebra homomorphisms
$$A \xrightarrow{i} B \xrightarrow{\pi} A$$
such that $\pi i$ is the identity and such that $B$ is a projective
$A$-module, then $\dim ( \underline{\mod} A)  \le \dim (
\underline{\mod} B)$.
\end{lemma}

\begin{proof}
Every $B$-module is an $A$-module via the map $i$. Moreover, if a
map in $\mod B$ factors through a projective $B$-module, then it
factors through a projective $A$-module since $B$ is $A$-projective.
Therefore the map $i$ induces a functor $F \colon \underline{\mod} B
\to \underline{\mod} A$, and this is clearly a functor of
triangulated categories. Now let $M$ be an object in
$\underline{\mod} A$. Then $M$ is a $B$-module via the map $\pi$,
hence every object in $\underline{\mod} A$ belongs to the image of
$F$. The result now follows from Proposition \ref{densesubcat}.
\end{proof}

\section{Quantum complete intersections}\label{quantumci}

Throughout the rest of this paper, let $k$ be a field and $n$ a
positive integer. Let $a_1, \dots, a_n$ be a sequence of integers
with $a_u \ge 2$, and for each pair $(i,j)$ of integers with $1 \le
i <j \le n$, let $q_{ij}$ be a nonzero element of $k$. Denote by
$\Lambda$ the algebra
$$\Lambda = k \langle X_1, \dots, X_n \rangle / (X_u^{a_u}, X_iX_j -
q_{ij}X_jX_i \text{ for } 1 \le i <j \le n ),$$ and by $x_i$ the
image of $X_i$ in $\Lambda$. This algebra is finite dimensional of
dimension $\prod a_i$, and a quantum complete intersection of
codimension $n$. Namely, it is the quotient of the quantum symmetric
algebra
$$k \langle X_1, \dots, X_n \rangle / (X_iX_j - q_{ij}X_jX_i \text{
for } 1 \le i <j \le n )$$ by the quantum regular sequence
$x_1^{a_1}, \dots, x_n^{a_n}$. Furthermore, these algebras are
Frobenius; the codimension two argument in the beginning of
\cite[Section 3]{Bergh} transfers to the general situation. In
particular, these algebras are selfinjective.

As mentioned in the introduction, the notion of quantum complete
intersections originates from work by Manin (cf.\ \cite{Manin}),
who introduced the concept of \emph{quantum symmetric algebras}.
These algebras were used by Avramov, Gasharov and Peeva in
\cite{Avramov} to study modules behaving homologically as modules
over commutative complete intersections. In particular, they
introduced \emph{quantum regular sequences} of endomorphisms of
modules, thus generalizing the classical notion of regular
sequences. In \cite{Benson} a rank variety theory was given for
quantum complete intersections satisfying certain conditions, and
in \cite{Bergh} the Hochschild cohomology and homology of these
algebras were studied.

We now prove the first of the main results. Given any subset $\{
i_1, \dots, i_t \}$ of $\{ 1, \dots, n \}$, let $\Lambda_{i_1,
\dots, i_t}$ denote the subalgebra of $\Lambda$ generated by
$x_{i_1}, \dots, x_{i_t}$. Thus $\Lambda_{i_1, \dots, i_t}$ is the
codimension $t$ quantum complete intersection we obtain when
``forgetting" the variables $X_i$ not in the sequence $X_{i_1},
\dots, X_{i_t}$.

\begin{theorem}\label{chain}
Let $\{ i_1, \dots, i_{n-1} \}$ be any subset of $\{ 1, \dots, n \}$
of size $n-1$. Then
$$\dim ( \underline{\mod} \Lambda_{i_1} ) \le \dim (
\underline{\mod} \Lambda_{i_1, i_2} ) \le \cdots \le \dim (
\underline{\mod} \Lambda_{i_1, \dots, i_{n-1}} ) \le \dim (
\underline{\mod} \Lambda ).$$
\end{theorem}

\begin{proof}
For any $t$ we have algebra homomorphisms
$$\Lambda_{i_1, \dots, i_{t-1}} \to \Lambda_{i_1, \dots, i_t} \to
\Lambda_{i_1, \dots, i_{n-1}}$$ whose composition is the identity.
Moreover, the middle term is free as a module over $\Lambda_{i_1,
\dots, i_{n-1}}$. Namely, it is isomorphic to
$\bigoplus_{j=0}^{a_{i_t}-1} ( \Lambda_{i_1, \dots, i_{n-1}}
)x_{i_t}^j$. The inequalities now follow from Lemma
\ref{subalgebra}.
\end{proof}

We end this section with a result giving an upper bound for the
representation dimension of a quantum complete intersection.

\begin{theorem}\label{upperbound}
The representation dimension of $\Lambda$ is at most $2n$.
\end{theorem}

\begin{proof}
We prove by induction on $n$ that if $\Lambda$ is a quantum complete
intersection of codimension $n$, then there exists a graded
$\Lambda$-module $M$ (we view $\Lambda$ as a $\mathbb{Z}^n$-graded
algebra) such that the following hold:
\begin{enumerate}
\item[(i)] $\Lambda$ is a direct summand of $M$
\item[(ii)] the global dimension of $\End_{\gr \Lambda}(M)$ is at most
$2n$
\item[(iii)] all the simple $\End_{\gr \Lambda}(M)$-modules are
one-dimensional.
\end{enumerate}
The theorem will then follow from (i) and (ii).

If $n=1$, then $\Lambda$ is of the form $k[X] / (X^a)$. In this
case, define $M$ by $M = \bigoplus_{i=1}^a k[X]/(X^i)$. Then it is
well known that (i),(ii) and (iii) hold for $M$. Next let $n \ge 1$,
and suppose that the above claim holds for every quantum complete
intersection of codimension at most $n$. Let $\Lambda$ be a
codimension $n+1$ quantum complete intersection
$$\Lambda = k \langle X_1, \dots, X_{n+1} \rangle / (X_u^{a_u}, X_iX_j -
q_{ij}X_jX_i \text{ for } 1 \le i <j \le n+1 ),$$ and let
$\Lambda_{1, \dots, n}$ and $\Lambda_{n+1}$ be the subalgebras
generated by $x_1, \dots, x_n$ and $x_{n+1}$, respectively. The
algebra $\Lambda_{1, \dots, n}$ is $\mathbb{Z}^n$-graded, whereas
$\Lambda_{n+1}$ is $\mathbb{Z}$-graded.

Define a group homomorphism $g \colon \mathbb{Z}^n \times \mathbb{Z}
\to k \setminus \{ 0 \}$ by
$$\left ( ( z_1, \dots, z_n ), z \right ) \mapsto \prod_{i=1}^n
q_{i,n+1}^{-zz_i},$$ and use this to define a ``twisted" tensor
product algebra $\Lambda_{1, \dots, n} \otimes^g_k \Lambda_{n+1}$ as
follows. As a $k$-vector space $\Lambda_{1, \dots, n} \otimes^g_k
\Lambda_{n+1}$ is just the usual tensor product $\Lambda_{1, \dots,
n} \otimes_k \Lambda_{n+1}$, but the multiplication is given by
$$( \lambda_1 \otimes \gamma_1 )( \lambda_2 \otimes \gamma_2 )
\stackrel{\text{def}}{=} g( | \lambda_2 |, | \gamma_1 | ) \lambda_1
\lambda_2 \otimes \gamma_1 \gamma_2,$$ where $\lambda_i$ and
$\gamma_i$ are homogeneous elements of $\Lambda_{1, \dots, n}$ and
$\Lambda_{n+1}$, respectively. With this ring structure we see that
the algebra $\Lambda_{1, \dots, n} \otimes^g_k \Lambda_{n+1}$ is
isomorphic to $\Lambda$.

By induction, there exist graded modules $M_1 \in \mod \Lambda_{1,
\dots, n}$ and $M_2 \in \mod \Lambda_{n+1}$ satisfying (i),(ii) and
(iii) over $\Lambda_{1, \dots, n}$ and $\Lambda_{n+1}$,
respectively. Consider the graded module $M_1 \otimes_k^g M_2$ over
$\Lambda_{1, \dots, n} \otimes^g_k \Lambda_{n+1}$, where the scalar
action is defined in the natural way. It is not hard to see that
$\Lambda_{1, \dots, n} \otimes^g_k \Lambda_{n+1}$ is a direct
summand of $M_1 \otimes_k^g M_2$, and so this module satisfies (i).
Now let $\Gamma$ be its endomorphism ring over $\Lambda_{1, \dots,
n} \otimes^g_k \Lambda_{n+1}$. Then $\Gamma \simeq \Gamma_1
\otimes_k^g \Gamma_2$, where $\Gamma_1 = \End_{\gr \Lambda_{1,
\dots, n}} (M_1)$ and $\Gamma_2 = \End_{\gr \Lambda_{n+1}}(M_2)$. If
$S$ is a simple $\Gamma$-module, then it is a quotient of a module
of the form $S_1 \otimes_k^g S_2$, where $S_i$ is a simple
$\Gamma_i$-module. Since $S_i$ is one-dimensional, the module $S$
must be isomorphic to $S_1 \otimes_k^g S_2$, giving
$$\Pd_{\Lambda_{1, \dots, n} \otimes^g_k \Lambda_{n+1}} S \le
\Pd_{\Lambda_{1, \dots, n}} S_1 + \Pd_{\Lambda_{n+1}} S_2 \le 2n +
2.$$ In particular, we see that (ii) and (iii) hold for the module
$M_1 \otimes_k^g M_2$. This completes the proof.
\end{proof}

\section{Homogeneous quantum complete intersections}

We now turn our attention to ``homogeneous" quantum complete
intersections of codimension $n$. Throughout the rest of this paper,
fix an integer $a \ge 2$ and a primitive $a$th root of unity $q \in
k$. We denote by $\Lambda_n^a$ the quantum complete intersection
$$\Lambda_n^a = \langle X_1, \dots, X_n \rangle / (X_j^a, X_iX_j -
qX_jX_i \text{ for } 1 \le i <j \le n ),$$ that is, all the defining
exponents and commutators are equal to $a$ and $q$, respectively.

Our aim is to show that the representation dimension of
$\Lambda_n^a$ is at least $n+1$. To do this, we show that given any
object $M \in \underline{\mod} \Lambda^a_n$, there are objects $\{
N_i \in \underline{\mod} \Lambda^a_n \}_{i=1}^n$ and morphisms $\{
N_i \xrightarrow{f_i} N_{i+1} \}_{i=1}^{n-1}$ satisfying the
following: for each $i$ the map
$$\underline{\Hom}_{\Lambda^a_n}(-,N_i) \xrightarrow{(f_i)_*}
\underline{\Hom}_{\Lambda^a_n}(-,N_{i+1})$$ vanishes on $\langle M
\rangle$, whereas the composition $(f_{n-1})_* \circ \cdots \circ
(f_1)_*$ does not vanish on $\underline{\mod} \Lambda^a_n$. We may
then conclude from \cite[Lemma 4.11]{Rouquier1} that $\dim (
\underline{\mod} \Lambda^a_n ) \ge n-1$, and then Proposition
\ref{lowerbound} gives $\repdim \Lambda^a_n \ge n+1$.

The first step is the following lemma on the behavior of a linear
combination of the generators. Whenever we have a tuple $\alpha = (
\alpha_1, \dots, \alpha_n ) \in k^n$, we denote the corresponding
element $\alpha_1 x_1 + \cdots + \alpha_n x_n \in \Lambda^a_n$ by
$\sigma_{\alpha}$.

\begin{lemma}\label{lincomb}
For any $n$-tuple $\alpha$ in $k^n$, the following hold.
\begin{enumerate}
\item[(i)] $\sigma_{\alpha}^a =0$.
\item[(ii)] $\sigma_{\alpha}^{a-1}x_i +
\sigma_{\alpha}^{a-2}x_i \sigma_{\alpha} + \cdots + \sigma_{\alpha}
x_i \sigma_{\alpha}^{a-2} + x_i \sigma_{\alpha}^{a-1} =0$.
\end{enumerate}
\end{lemma}

\begin{proof}
The first part of the lemma is just \cite[Lemma 2.3]{Benson}. To
prove the second part, assume that $k$ is infinite, and let
$\epsilon$ be any nonzero element of $k$. Then from (i) the element
$( \sigma_{\alpha} + \epsilon x_i )^a$ is zero, and expanding we get
$$0 = \sigma_{\alpha}^a + \epsilon \left ( \sum_{j=0}^{a-1} \sigma_{\alpha}^j
x_i \sigma_{\alpha}^{a-1-j} \right ) + \epsilon^2 z,$$ where $z$ is
some element in $\Lambda_n^a$ depending on $\epsilon$. Therefore the
set of all elements $\epsilon$ in $k$ for which the equality
$$0 = \epsilon z + \sum_{j=0}^{a-1} \sigma_{\alpha}^j x_i
\sigma_{\alpha}^{a-1-j}$$ holds contains all the nonzero elements.
However, this set is closed in the Zariski topology, and so since
$k$ is infinite the zero element must also belong to it. This proves
the second part of the lemma in the case when $k$ is infinite. If
$k$ is finite, let $K / k$ be an infinite field extension. Then $K
\otimes_k \Lambda_n^a$ is the quantum complete intersection we
obtain when allowing the scalars to be elements of $K$, with the
same defining relations as in $\Lambda_n^a$. Since (ii) holds in $K
\otimes_k \Lambda_n^a$ by the above, it also holds in $\Lambda_n^a$.
\end{proof}

In order to prove the main result, we also need the following two
lemmas. These are versions of \cite[Lemma 9]{Oppermann1} and
\cite[Proposition 11]{Oppermann1}, respectively.

\begin{lemma}\label{openset}
Suppose $k$ is infinite, and let $M \in \mod \Lambda^a_n$ be a
module. Then there exists a non-empty open subset $U_M \subseteq
k^n$ such that for any $\alpha \in U_M$ and any $m \in M$, the
following implications hold:
\begin{eqnarray*}
\sigma_{\alpha} m =0 & \Rightarrow & \sigma_{\beta} m \in
\sigma_{\alpha} M \text{ for all } \beta \in k^n \\
\sigma_{\alpha}^{a-1} m =0 & \Rightarrow & \left ( \sum_{i=0}^{a-2}
\sigma_{\alpha}^i \sigma_{\beta} \sigma_{\alpha}^{a-2-i} \right ) m
\in \sigma_{\alpha}^{a-1} M \text{ for all } \beta \in k^n.
\end{eqnarray*}
\end{lemma}

\begin{proof}
Let $U_1$ be the set of all elements $\alpha \in k^n$ such that the
matrix representing the linear map $M \xrightarrow{\sigma_{\alpha}}
M$ has maximal rank. Then $U_1$ is non-empty and open, by an
argument similar to that of the proof of \cite[Lemma 9]{Oppermann1}.
For an element $\alpha \in U_1$, choose a basis
$$\{ m_1, \dots, m_s, w_1, \dots, w_t \}$$
of $M$ such that $\sigma_{\alpha} m_i =0$ and the set $\{
\sigma_{\alpha} w_1, \dots, \sigma_{\alpha} w_t \}$ is linearly
independent. Then for any nonzero $\epsilon \in k$ and any $\beta
\in k^n$, the set
$$\{ \sigma_{( \alpha + \epsilon \beta )} m_i, \sigma_{( \alpha + \epsilon
\beta )} w_1, \dots, \sigma_{( \alpha + \epsilon \beta )} w_t \}$$
is linearly dependent by the choice of $\alpha$. Since
$\sigma_{\alpha}m_i =0$, the set
$$\{ \sigma_{\beta} m_i, \sigma_{( \alpha + \epsilon
\beta )} w_1, \dots, \sigma_{( \alpha + \epsilon \beta )} w_t \}$$
is also linearly dependent for any nonzero $\epsilon \in k$ and any
$\beta \in k^n$. However, the set of all $\epsilon \in k$ such that
this set is linearly dependent is closed, and since it contains all
the nonzero elements it must be $k$ itself. Consequently, for any
element $\beta \in k^n$ the set
$$\{ \sigma_{\beta} m_i, \sigma_{\alpha} w_1, \dots, \sigma_{\alpha} w_t
\}$$ is linearly dependent, i.e.\  $\sigma_{\beta} m_i \in
\sigma_{\alpha} M$. Therefore, given any elements $\alpha \in U_1$
and $m \in M$, the first implication in the lemma holds.

Next define $U_2$ to be the set of all elements $\alpha \in k^n$
such that the matrix representing the linear map $M
\xrightarrow{\sigma_{\alpha}^{a-1}} M$ has maximal rank. Then $U_2$
is non-empty and open, and an argument similar to that for $U_1$,
using Lemma \ref{lincomb}(ii), shows that given any elements $\alpha
\in U_2$ and $m \in M$, the second implication in the lemma holds.
Now since $U_1$ and $U_2$ are non-empty open sets in the Zariski
topology, their intersection $U_M = U_1 \cap U_2$ is also non-empty
and open, and this is a set having the properties we are seeking.
\end{proof}

\begin{lemma}\label{factoring}
Suppose $k$ is infinite, and let $M$ be a $\Lambda^a_n$-module. Then
there exists a non-empty open subset $U \subseteq k^n$ such that for
any $\alpha \in U$ and any $1 \le p \le n$, the following hold.
\begin{enumerate}
\item[(i)] For every $j \in \mathbb{Z}$, any composition
$$\Lambda^a_n / ( \sigma_{\alpha} ) \xrightarrow{\cdot \sum_{i=0}^{a-2}
\sigma_{\alpha}^i x_p \sigma_{\alpha}^{a-2-i}} \Lambda^a_n / (
\sigma_{\alpha}^{a-1} ) \to \Omega_{\Lambda^a_n}^j(M)$$ of left
$\Lambda^a_n$-homomorphisms is zero in $\underline{\mod}
\Lambda^a_n$.
\item [(ii)] For every $j \in \mathbb{Z}$, any composition
$$\Lambda^a_n / ( \sigma_{\alpha}^{a-1} ) \xrightarrow{\cdot
x_p} \Lambda^a_n / ( \sigma_{\alpha} ) \to
\Omega_{\Lambda^a_n}^j(M)$$ of left $\Lambda$-homomorphisms is zero
in $\underline{\mod} \Lambda^a_n$.
\end{enumerate}
\end{lemma}

\begin{proof}
\sloppy For simplicity, we denote our algebra $\Lambda^a_n$ by just
$\Lambda$. First note that the maps $\Lambda / ( \sigma_{\alpha} )
\xrightarrow{\cdot \sum_{i=0}^{a-2} \sigma_{\alpha}^i x_p
\sigma_{\alpha}^{a-2-i}} \Lambda / ( \sigma_{\alpha}^{a-1} )$ and
$\Lambda / ( \sigma_{\alpha}^{a-1} ) \xrightarrow{\cdot x_p} \Lambda
/ ( \sigma_{\alpha} )$ of left modules are well defined by Lemma
\ref{lincomb}(ii). Choose two non-empty open subsets $U_M$ and
$U_{\Omega_{\Lambda}(M)}$ of $k^n$ satisfying Lemma \ref{openset},
and consider their intersection $U = U_M \cap
U_{\Omega_{\Lambda}(M)}$. Then $U$ is also non-empty and open. Take
any $\alpha \in U$ and let $\Lambda / ( \sigma_{\alpha}^{a-1} )
\xrightarrow{g} N$ be a homomorphism of left $\Lambda$-modules,
where $N$ is either $M$ or $\Omega_{\Lambda}(M)$. Furthermore, fix
any $1 \le p \le n$, and let $m = g( 1 + ( \sigma_{\alpha}^{a-1} )
)$. Then $\sigma_{\alpha}^{a-1} m =0$, and so by Lemma \ref{openset}
there exists an element $m' \in N$ with the property that $\left (
\sum_{i=0}^{a-2} \sigma_{\alpha}^i x_p \sigma_{\alpha}^{a-2-i}
\right ) m = \sigma_{\alpha}^{a-1} m'$. This gives the factorization
$$\xymatrix@C=70pt{
\Lambda / ( \sigma_{\alpha} ) \ar[r]^{\cdot \sum_{i=0}^{a-2}
\sigma_{\alpha}^i x_p \sigma_{\alpha}^{a-2-i}} \ar[dr]_{1 \mapsto
\sigma_{\alpha}^{a-1}} & \Lambda / ( \sigma_{\alpha}^{a-1} )
\ar[r]^g & N \\
& \Lambda \ar[ur]_{1 \mapsto m'} }$$ proving (i) in the case when
$j=0,1$. A similar argument shows that (ii) holds for $j=0,1$.

Consider the two diagrams
$$\xymatrix@C=50pt{
0 \ar[r] & \Lambda / ( \sigma_{\alpha} ) \ar[d]^{\cdot
\sum_{i=0}^{a-2} \sigma_{\alpha}^i x_p \sigma_{\alpha}^{a-2-i}}
\ar[r]^{\cdot \sigma_{\alpha}^{a-1}} & \Lambda \ar[d]^{\cdot x_p}
\ar[r] & \Lambda / ( \sigma_{\alpha}^{a-1} ) \ar[d]^{\cdot x_p}
\ar[r] & 0 \\
0 \ar[r] & \Lambda / ( \sigma_{\alpha}^{a-1} ) \ar[r]^{\cdot ( -
\sigma_{\alpha})} & \Lambda \ar[r] & \Lambda / ( \sigma_{\alpha} )
\ar[r] & 0 \\
0 \ar[r] & \Lambda / ( \sigma_{\alpha}^{a-1} ) \ar[d]^{\cdot x_p}
\ar[r]^{\cdot (- \sigma_{\alpha})} & \Lambda \ar[d]^{\cdot
\sum_{i=0}^{a-2} \sigma_{\alpha}^i x_p \sigma_{\alpha}^{a-2-i}}
\ar[r] & \Lambda / ( \sigma_{\alpha} ) \ar[d]^{\cdot
\sum_{i=0}^{a-2} \sigma_{\alpha}^i x_p \sigma_{\alpha}^{a-2-i}}
\ar[r] & 0 \\
0 \ar[r] & \Lambda / ( \sigma_{\alpha} ) \ar[r]^{\cdot
\sigma_{\alpha}^{a-1}} & \Lambda \ar[r] & \Lambda / (
\sigma_{\alpha}^{a-1} ) \ar[r] & 0 }$$ which are commutative by
Lemma \ref{lincomb}(ii). Since the rows are exact, the modules
$\Lambda / ( \sigma_{\alpha} )$ and $\Lambda / (
\sigma_{\alpha}^{a-1} )$ are both $2$-periodic with respect to the
syzygy operator, i.e.\ $\Omega_{\Lambda}^{2u} ( \Lambda / (
\sigma_{\alpha} ) ) \simeq \Lambda / ( \sigma_{\alpha} )$ and
$\Omega_{\Lambda}^{2u} ( \Lambda / ( \sigma_{\alpha}^{a-1} ) )
\simeq \Lambda / ( \sigma_{\alpha}^{a-1} )$ for any $u \in
\mathbb{Z}$. Moreover, the commutativity of the squares shows that
the two maps $\Lambda / ( \sigma_{\alpha} ) \xrightarrow{\cdot
\sum_{i=0}^{a-2} \sigma_{\alpha}^i x_p \sigma_{\alpha}^{a-2-i}}
\Lambda / ( \sigma_{\alpha}^{a-1} )$ and $\Lambda / (
\sigma_{\alpha}^{a-1} ) \xrightarrow{\cdot x_p} \Lambda / (
\sigma_{\alpha} )$ are also $2$-periodic with respect to the syzygy
operator. Now let $\Lambda / ( \sigma_{\alpha}^{a-1} )
\xrightarrow{g} \Omega_{\Lambda}^j(M)$ and $\Lambda / (
\sigma_{\alpha} ) \xrightarrow{f} \Omega_{\Lambda}^j(M)$ be any
maps, and choose an integer $u \in \mathbb{Z}$ such that
$\Omega_{\Lambda}^{j+2u}(M)$ is isomorphic to either $M$ or
$\Omega_{\Lambda}(M)$ in $\underline{\mod} \Lambda$. Then the
diagrams
$$\xymatrix@C=60pt{
\Lambda / ( \sigma_{\alpha} ) \ar[r]^{\cdot \sum_{i=0}^{a-2}
\sigma_{\alpha}^i x_p \sigma_{\alpha}^{a-2-i}} \ar[d]^{\wr} &
\Lambda / ( \sigma_{\alpha}^{a-1} ) \ar[r]^g \ar[d]^{\wr} &
\Omega_{\Lambda}^j(M) \\
\Lambda / ( \sigma_{\alpha} ) \ar[r]^{\cdot \sum_{i=0}^{a-2}
\sigma_{\alpha}^i x_p \sigma_{\alpha}^{a-2-i}} & \Lambda / (
\sigma_{\alpha}^{a-1} ) \ar[r]^{\Omega_{\Lambda}^{2u}(g)} &
\Omega_{\Lambda}^{j+2u}(M) \\
\Lambda / ( \sigma_{\alpha}^{a-1} ) \ar[r]^{\cdot x_p} \ar[d]^{\wr}
& \Lambda / ( \sigma_{\alpha} ) \ar[r]^f \ar[d]^{\wr} &
\Omega_{\Lambda}^j(M) \\
\Lambda / ( \sigma_{\alpha}^{a-1} ) \ar[r]^{\cdot x_p} & \Lambda / (
\sigma_{\alpha} ) \ar[r]^{\Omega_{\Lambda}^{2u}(f)} &
\Omega_{\Lambda}^{j+2u}(M) }$$ are commutative in $\underline{\mod}
\Lambda$, where the vertical maps are isomorphisms. We have already
shown that the bottom compositions are zero in $\underline{\mod}
\Lambda$, but then so are the top compositions, since these are
shifts of the bottom compositions.
\end{proof}

We are now ready to prove that the representation dimension of
$\Lambda_n^a$ is at least $n+1$ when $n$ is even and the field $k$
is infinite.

\begin{proposition}\label{infinitefieldeven}
If $k$ is infinite and $n$ is even, then $\repdim \Lambda_n^a \ge
n+1$.
\end{proposition}

\begin{proof}
As in the previous proof, we denote our algebra $\Lambda_n^a$ by
just $\Lambda$. Let $\alpha \in k^n$ be an $n$-tuple, and denote the
element
$$x_{n-1} \left ( \sum_{i=0}^{a-2} \sigma_{\alpha}^i x_n \sigma_{\alpha}^{a-2-i}
\right ) x_{n-3} \left ( \sum_{i=0}^{a-2} \sigma_{\alpha}^i x_{n-2}
\sigma_{\alpha}^{a-2-i} \right ) \cdots x_3 \left ( \sum_{i=0}^{a-2}
\sigma_{\alpha}^i x_4 \sigma_{\alpha}^{a-2-i} \right ) x_2$$ by
$w_{\alpha}$. In the first (and longest) part of this proof, we show
that the set of all $\alpha \in k^n$ such that $w_{\alpha}$ does not
belong to $\sigma_{\alpha} \Lambda + \Lambda \sigma_{\alpha}$
contains a non-empty open set.

Fix any $n$-tuple $\alpha = ( \alpha_1, \dots, \alpha_n ) \in k^n$
with $\alpha_1 \neq 0$, and, for simplicity, denote the
corresponding $\sigma_{\alpha}$ and $w_{\alpha}$ by $\sigma$ and
$w$, respectively. Every element $\lambda \in \Lambda$ admits a
unique decomposition $\lambda = N_{x_1}( \lambda ) + \sigma R_{x_1}
( \lambda )$, in which $x_1$ does not occur in any of the monomials
in $N_{x_1}( \lambda )$. With this notation, we see that
\begin{eqnarray*}
w \in \sigma \Lambda + \Lambda \sigma & \Leftrightarrow &
N_{x_1}(w) \in \sigma \Lambda + \Lambda \sigma \\
& \Leftrightarrow & N_{x_1}(w) + h \sigma \in \sigma \Lambda
\text{ for some } h \in \Lambda \\
& \Leftrightarrow & N_{x_1}(w) + h \sigma \in \sigma \Lambda
\text{ for some } h \in \Lambda \text{ with } N_{x_1}(h)=h \\
& \Leftrightarrow & N_{x_1} \left ( N_{x_1}(w) + h \sigma \right )
=0 \text{ for some } h \in \Lambda \text{ with } N_{x_1}(h)=h,
\end{eqnarray*}
where we may assume that the $h$ occurring is homogeneous. Now note
that the degree of $w$ is $\left ( \frac{n}{2} - 1 \right ) a+1$.
This implies that any homogeneous $h$ satisfying the above
implications is of degree $\left ( \frac{n}{2} - 1 \right ) a$. If
in addition $N_{x_1}(h)=h$, i.e.\ if $x_1$ does not occur in any of
the monomials in $h$, then $hx_1 = q^{- \left ( \frac{n}{2} - 1
\right ) a}x_1h = x_1h$. Writing $\sigma' = \sigma - \alpha_1 x_1$
we then get
$$h \sigma = \alpha_1hx_1 + h \sigma' = \alpha_1x_1 h +
h \sigma'  = \sigma h - \sigma' h + h \sigma'.$$ What we have so
far gives
\begin{eqnarray*}
w \in \sigma \Lambda + \Lambda \sigma & \Leftrightarrow & N_{x_1}
\left ( N_{x_1}(w) + h \sigma \right ) =0 \text{ with } N_{x_1}(h)=h \\
& \Leftrightarrow & N_{x_1} \left ( N_{x_1}(w) + \sigma h -
\sigma' h + h \sigma' \right ) =0 \text{ with } N_{x_1}(h)=h \\
& \Leftrightarrow & N_{x_1} \left ( N_{x_1}(w) - \sigma' h + h
\sigma' \right ) =0 \text{ with } N_{x_1}(h)=h  \\
& \Leftrightarrow & N_{x_1}(w) - \sigma' h + h \sigma' =0 \text{
with } N_{x_1}(h)=h,
\end{eqnarray*}
where the last implication follows from the fact that $x_1$ does not
occur in any of the monomials in $N_{x_1}(w) - \sigma' h + h
\sigma'$. Thus $w$ belongs to $\sigma \Lambda + \Lambda \sigma$ if
and only if $N_{x_1}(w) = \sigma' h - h \sigma'$ for some
homogeneous element $h$ satisfying $N_{x_1}(h)=h$.

Let $\lambda$ be the element $x_2x_3^{a-1}x_4 \cdots
x_{n-1}^{a-1}x_n$, which has the same degree as $w$. For any
homogeneous element $h$ of degree one less than that of $w$, the
coefficient of $\lambda$ in $\sigma' h - h \sigma'$ is easily seen
to be zero. Therefore $w \notin \sigma \Lambda + \Lambda \sigma$ if
the coefficient of $\lambda$ in $N_{x_1}(w)$ is nonzero. Note that
the set of tuples $(\alpha_2, \dots, \alpha_n) \in k^{n-1}$ for
which this holds is open.

Consider the tuple $( \alpha_1,0,1, \dots, 1,0) \in k^n$, that is,
we take as $\sigma$ the element
$$\sigma = \alpha_1x_1 + x_3 + \cdots + x_{n-1}.$$
Define the element $w' \in \Lambda$ by
$$w' = x_{n-1} \sum_{i=0}^{a-2} \sigma^i x_n \sigma^{a-2-i},$$
i.e.\ $w'$ is the ``first part" of $w$. Since $x_n \sigma = q^{-1}
\sigma x_n$, we see that
\begin{eqnarray*}
w' & = & x_{n-1} \beta \sigma^{a-2} x_n \\
& = & \beta ( q^{-1} \alpha_1x_1 + \cdots + q^{-1}x_{n-3} + x_{n-1}
)^{a-2}x_{n-1}x_n
\end{eqnarray*}
for some nonzero element $\beta \in k$. Using the equality
$N_{x_1}(zy) = N_{x_1} \left ( N_{x_1}(z)y \right )$, which holds
for any elements $z,y \in \Lambda$, an induction argument gives
$$N_{x_1} \left ( ( q^{-1} \alpha_1x_1 + \cdots + q^{-1}x_{n-3} + x_{n-1}
)^i \right ) = \left ( \prod_{j=1}^i (1-q^{-j}) \right )
x_{n-1}^i.$$ Applying $N_{x_1}$ to the above expression for $w'$
then gives
\begin{eqnarray*}
N_{x_1}( w' ) &=& \beta N_{x_1} \left ( N_{x_1} \left [ ( q^{-1}
\alpha_1x_1 + \cdots + q^{-1}x_{n-3} + x_{n-1} )^{a-2} \right ]
x_{n-1}x_n
\right ) \\
&=& \beta N_{x_1} \left ( \left [ \prod_{j=1}^{a-2} (1-q^{-j})
\right ] x_{n-1}^{a-1}x_n \right ) \\
&=& \beta' x_{n-1}^{a-1}x_n,
\end{eqnarray*}
where $\beta'$ is some nonzero element in $k$. Now define $w''$ by
$$w'' =x_{n-3} \left ( \sum_{i=0}^{a-2} \sigma^i x_{n-2}
\sigma^{a-2-i} \right ) \cdots x_3 \left ( \sum_{i=0}^{a-2} \sigma^i
x_4 \sigma^{a-2-i} \right ) x_2,$$ that is, the element $w$ is given
by $w=w'w''$. By what we have just shown, the equality
$$N_{x_1}(w) = N_{x_1} \left ( N_{x_1} (w')w'' \right )  = \beta' N_{x_1}
(x_{n-1}^{a-1}x_nw'')$$ holds. Now every monomial in $w''$
containing $x_{n-1}$ does not ``contribute" to
$x_{n-1}^{a-1}x_nw''$, therefore we may replace every $\sigma$ in
$w''$ by $\sigma - x_{n-1}$. Then if we repeat the above process, we
see that
$$N_{x_1}(w) = \delta x_2x_3^{a-1}x_4 \cdots
x_{n-1}^{a-1}x_n = \delta \lambda,$$ where $\delta$ is some nonzero
element in $k$. This shows that the coefficient of the monomial
$\lambda$ in $N_{x_1}(w)$ is nonzero, i.e.\ $w \notin \sigma \Lambda
+ \Lambda \sigma$.

Define the set $V \subseteq k^n$ by
$$V = \{ \alpha \in k^n \mid w_{\alpha} \notin \sigma_{\alpha}
\Lambda + \Lambda \sigma_{\alpha} \},$$ and consider the set
$$W = \{ ( \alpha_2, \dots, \alpha_n ) \in k^{n-1} \mid ( \alpha_1, \alpha_2,
\dots, \alpha_n) \in V \text{ for all } 0 \neq \alpha_1 \in k \}.$$
We have just shown that $W$ is open and non-empty, hence the subset
$$\{ ( \alpha_1, \alpha_2, \dots, \alpha_n) \in k^n \mid \alpha_1
\neq 0 \text{ and } ( \alpha_2, \dots, \alpha_n ) \in W \}$$ of $V$
is non-empty and open in $k^n$.

Now let $M$ be a $\Lambda$-module, and let $U \subseteq k^n$ be a
non-empty open set having the properties stated in Lemma
\ref{factoring}. Since $V$ contains a non-empty open subset, the
intersection $U \cap V$ is not empty, and therefore contains an
$n$-tuple $\alpha \in k^n$. For each $1 \le i \le n-1$, define maps
$f_i \in \underline{\mod} \Lambda$ by
$$f_i = \left \{
\begin{array}{ll}
\Lambda / ( \sigma_{\alpha}^{a-1} ) \xrightarrow{\cdot x_{n-i}}
\Lambda / ( \sigma_{\alpha} ) & \text{ when } i \text{ is odd and }
i \le n-3 \\
\Lambda / ( \sigma_{\alpha} ) \xrightarrow{\cdot \sum_{i=0}^{a-2}
\sigma_{\alpha}^i x_{n+2-i} \sigma_{\alpha}^{a-2-i}} \Lambda / (
\sigma_{\alpha}^{a-1} ) & \text{ when } i \text{ is even} \\
\Lambda / ( \sigma_{\alpha}^{a-1} ) \xrightarrow{\cdot x_2} \Lambda
/ ( \sigma_{\alpha} ) & \text{ when } i=n-1.
\end{array} \right.$$
By Lemma \ref{factoring} each induced map $(f_i)_*$ vanishes on the
subcategory $\langle M \rangle$ of $\underline{\mod} \Lambda$.
However, the composition $(f_{n-1})_* \circ \cdots \circ (f_1)_*$
does not vanish: if it did then the map $\Lambda / (
\sigma_{\alpha}^{a-1} ) \xrightarrow{\cdot w_{\alpha}} \Lambda / (
\sigma_{\alpha} )$ would be zero in $\underline{\mod} \Lambda$.
However, the injective envelope of $\Lambda / (
\sigma_{\alpha}^{a-1} )$ is $\Lambda / ( \sigma_{\alpha}^{a-1} )
\xrightarrow{\cdot \sigma_{\alpha}} \Lambda$, whereas the projective
cover of $\Lambda / ( \sigma_{\alpha} )$ is the projection $\Lambda
\xrightarrow{\pi} \Lambda / ( \sigma_{\alpha} )$. Therefore, if the
map $\Lambda / ( \sigma_{\alpha}^{a-1} ) \xrightarrow{\cdot
w_{\alpha}} \Lambda / ( \sigma_{\alpha} )$ was zero in
$\underline{\mod} \Lambda$, then it would factor in a commutative
diagram
$$\xymatrix{
\Lambda / ( \sigma_{\alpha}^{a-1} ) \ar[d]^{\cdot \sigma_{\alpha}}
\ar[r]^{\cdot w_{\alpha}} & \Lambda / ( \sigma_{\alpha} ) \\
\Lambda \ar[r] & \Lambda \ar[u]^{\pi} }$$ in $\mod \Lambda$. Since
any homomorphism $\Lambda \to \Lambda$ is right multiplication with
an element in $\Lambda$, this would mean that there exists an
element $\lambda \in \Lambda$ such that $w_{\alpha} =
\sigma_{\alpha} \lambda$ in $\Lambda / ( \sigma_{\alpha} )$. In
other words, the element $w_{\alpha}$ would belong to
$\sigma_{\alpha} \Lambda + \Lambda \sigma_{\alpha}$, and this is not
the case since $\alpha \in V$. From \cite[Lemma 4.11]{Rouquier1} we
see that $\dim ( \underline{\mod} \Lambda ) \ge n-1$, and so
Proposition \ref{lowerbound} gives $\repdim \Lambda \ge n+1$.
\end{proof}

\begin{remark}
It is also possible to prove Proposition \ref{infinitefieldeven}
using \cite[Theorem 1(b)]{Oppermann2}. However, as in the proof we
just gave, the key ingredient in such a proof is deciding whether or
not $w_{\alpha}$ belongs to $\sigma_{\alpha} \Lambda + \Lambda
\sigma_{\alpha}$.
\end{remark}

Having established the case when $n$ is even and the field $k$ is
infinite, we prove the main theorem in this paper, namely the
general case.

\begin{theorem}\label{infinitefieldodd}
The representation dimension of $\Lambda_n^a$ satisfies
$$n+1 \le \repdim \Lambda_n^a \le 2n.$$
\end{theorem}

\begin{proof}
Suppose first that $k$ is infinite. Let $n$ be even, and denote our
algebra $\Lambda_n^a$ by $\Lambda$. Furthermore, denote by
$\tilde{\Lambda}$ the algebra $\Lambda_{2,3,\dots,n}$, where the
notation is the same as that used in Theorem \ref{chain}. In other
words, the algebra $\tilde{\Lambda}$ is the codimension $n-1$
quantum complete intersection subalgebra of $\Lambda$ generated by
$x_2, \dots, x_n$. Our first aim is to show that $\repdim
\tilde{\Lambda} \ge n$. We may assume $n \ge 4$, since we know that
the representation dimension of the truncated polynomial ring is
two.

Given an $n$-tuple $\alpha \in k^n$, define as before the
corresponding element $\sigma_{\alpha} \in \Lambda$, and denote by
$w_{\alpha}$ the element
$$x_{n-1} \left ( \sum_{i=0}^{a-2} \sigma_{\alpha}^i x_n \sigma_{\alpha}^{a-2-i}
\right ) x_{n-3} \left ( \sum_{i=0}^{a-2} \sigma_{\alpha}^i x_{n-2}
\sigma_{\alpha}^{a-2-i} \right ) \cdots x_3 \left ( \sum_{i=0}^{a-2}
\sigma_{\alpha}^i x_4 \sigma_{\alpha}^{a-2-i} \right ) x_2.$$ We
showed in the proof of the previous result that the set
$$V = \{ \alpha \in k^n \mid w_{\alpha} \notin
\sigma_{\alpha} \Lambda + \Lambda \sigma_{\alpha} \}$$ contains a
non-empty open subset. Now for any $(n-1)$-tuple $\tilde{\alpha} =(
\alpha_2, \dots, \alpha_n ) \in k^{n-1}$, define the element
$\tilde{\sigma}_{\tilde{\alpha}} \in \tilde{\Lambda}$ by
$\tilde{\sigma}_{\tilde{\alpha}} = \alpha_2x_2 + \cdots +
\alpha_nx_n$, and denote by $\tilde{w}_{\tilde{\alpha}}$ the element
$$x_{n-1} \left ( \sum_{i=0}^{a-2} \tilde{\sigma}_{\tilde{\alpha}}^i x_n
\tilde{\sigma}_{\tilde{\alpha}}^{a-2-i}
\right ) x_{n-3} \left ( \sum_{i=0}^{a-2}
\tilde{\sigma}_{\tilde{\alpha}}^i x_{n-2}
\tilde{\sigma}_{\tilde{\alpha}}^{a-2-i} \right ) \cdots x_3 \left (
\sum_{i=0}^{a-2} \tilde{\sigma}_{\tilde{\alpha}}^i x_4
\tilde{\sigma}_{\tilde{\alpha}}^{a-2-i} \right ) $$ in
$\tilde{\Lambda}$. Furthermore, consider the subset
$$\tilde{V} = \{ \tilde{\alpha} \in k^{n-1} \mid
\tilde{w}_{\tilde{\alpha}} \notin \tilde{\sigma}_{\tilde{\alpha}}
\tilde{\Lambda} + \tilde{\Lambda}
\tilde{\sigma}_{\tilde{\alpha}}^{a-1} \}$$ of $k^{n-1}$. We show
that this set contains a non-empty subset which is open.

Recall from the previous proof that the set
$$W = \{ ( \alpha_2, \dots, \alpha_n ) \in k^{n-1} \mid ( \alpha_1, \alpha_2,
\dots, \alpha_n) \in V \text{ for all } 0 \neq \alpha_1 \in k \}$$
is non-empty and open in $k^{n-1}$. Furthermore, let $W'$ be the set
of all $(n-1)$-tuples $( \alpha_2, \dots, \alpha_n )$ in which
$\alpha_2$ is nonzero. Since $W'$ is non-empty and open, the
intersection $\tilde{W} = W \cap W'$ is also non-empty and open.
Pick therefore an element $\tilde{\alpha} = ( \alpha_2, \dots,
\alpha_n ) \in \tilde{W}$, let $\alpha \in k^n$ be the $n$-tuple
$\alpha = ( \alpha_2, \alpha_2, \dots, \alpha_n )$, and define a map
$f \colon \tilde{\Lambda} \to \Lambda$ by
$$x_i \mapsto \left \{
\begin{array}{ll}
x_1 + x_2 & \text{when } i=2 \\
x_i & \text{when } i \neq 2.
\end{array}
\right.$$ Note that $f( \tilde{\sigma}_{\tilde{\alpha}}) =
\sigma_{\alpha}$ and that $\alpha$ belongs to $V$. Moreover, note
that $f( \tilde{w}_{\tilde{\alpha}}) x_2 = w_{\alpha}$, and so if
$\tilde{w}_{\tilde{\alpha}}$ belongs to
$\tilde{\sigma}_{\tilde{\alpha}} \tilde{\Lambda} + \tilde{\Lambda}
\tilde{\sigma}_{\tilde{\alpha}}^{a-1}$ then $w_{\alpha}$ belongs to
$\sigma_{\alpha} \Lambda + \Lambda \sigma_{\alpha}^{a-1} x_2$.
However, from Lemma \ref{lincomb}(ii) we see that $\sigma_{\alpha}
\Lambda + \Lambda \sigma_{\alpha}^{a-1} x_2$ is contained in
$\sigma_{\alpha} \Lambda + \Lambda \sigma_{\alpha}$, and we know
that $w_{\alpha} \notin \sigma_{\alpha} \Lambda + \Lambda
\sigma_{\alpha}$ since $\alpha \in V$. Therefore
$\tilde{w}_{\tilde{\alpha}}$ cannot belong to
$\tilde{\sigma}_{\tilde{\alpha}} \tilde{\Lambda} + \tilde{\Lambda}
\tilde{\sigma}_{\tilde{\alpha}}^{a-1}$, and this shows that
$\tilde{W}$ is a subset of $\tilde{V}$. Consequently, the set
$\tilde{V}$ contains a non-empty subset which is open.

The arguments we applied at the end of the previous proof now shows
that $\repdim \tilde{\Lambda} \ge n$, and we have therefore proved
that $\repdim \Lambda_n^a \ge n+1$ when the field $k$ is infinite.
However, when $k$ is finite, the strategy applied in \cite[Section
4]{Oppermann1} carries over to our algebra and shows that the
inequality still holds. Therefore $\repdim \Lambda_n^a \ge n+1$
regardless of whether the field $k$ is infinite. This proves the
first inequality in the theorem, the other is Theorem
\ref{upperbound}.
\end{proof}

\begin{corollary}\label{subalg}
Let $\Lambda$ be a general quantum complete intersection, i.e.\
$$\Lambda = k \langle X_1, \dots, X_n \rangle / (X_u^{a_u}, X_iX_j -
q_{ij}X_jX_i \text{ for } 1 \le i <j \le n )$$ where $a_u \ge 2$ and
$q_{ij}$ is nonzero. If there exists a subset $\{ i_1, \dots, i_t
\}$ of $\{ 1, \dots, n \}$ such that the subalgebra $\Lambda_{i_1,
\dots, i_t}$ is a homogeneous quantum complete intersection of the
form $\Lambda_t^a$, then $\repdim \Lambda \ge t+1$.
\end{corollary}

\begin{proof}
By Theorem \ref{chain} the inequality  $\dim ( \underline{\mod}
\Lambda_t^a ) \le \dim ( \underline{\mod} \Lambda )$ holds, and in
the previous proof we showed that $\dim ( \underline{\mod}
\Lambda_t^a )$ is at least $t-1$. Proposition \ref{lowerbound} now
gives $\repdim \Lambda \ge t+1$.
\end{proof}

\begin{remarks}
(i) In the main results we have only considered homogeneous quantum
complete intersections of the form
$$k \langle X_1, \dots, X_n \rangle / (X_i^a, X_iX_j -
qX_jX_i \text{ for } 1 \le i <j \le n ),$$ that is, algebras where
the defining exponents are all equal to $a$ and where $q$ is a
primitive $a$th root of unity. However, the proofs are also valid if
we relax the requirement that the defining exponents of the ``even"
indeterminates $X_2, X_4, \dots$ are equal to $a$, as long as we
require that these exponents belong to $\{ 2, \dots, a \}$. This is
because the key ingredient in the proof when $n$ is even is to show
that the coefficient of the element $x_2x_3^{a-1}x_4 \cdots
x_{n-1}^{a-1}x_n$ is nonzero in a certain homogeneous element.

(ii) Work in progress by Avramov and Iyengar (cf.\ \cite{Avramov2})
shows that the dimension of the stable derived category of a
commutative local complete intersection of codimension $c$ is at
least $c-1$. Consequently, the representation dimension of an Artin
complete intersection is strictly greater than its embedding
dimension. In particular, they have shown that the representation
dimension of $k[X_1, \dots, X_n]/(X_1^{a_1}, \dots, X_n^{a_n})$ is
at least $n+1$ when $a_i \ge 2$.
\end{remarks}

\end{document}